\documentclass{amsart} 

\title{Bushy pseudocharacters and group actions on quasi-trees.}
\date{}

\author[Á.~Martínez-Pérez]{Álvaro Martínez-Pérez}
\address{Departamento de Geometría y Topología\\ Universidad Complutense de Madrid\\ Madrid 28040  Spain}
\email{alvaro\_martinez@mat.ucm.es}
       
\thanks{Partially supported by MTM-2009-07030.}

\thanks{The author would like to express his
gratitude to the mathematics department of the UIC for their hospitality 
and to Kevin Whyte for sugesting this work.}

\usepackage{amsfonts}
\usepackage{amsmath}
\usepackage{latexsym}
\usepackage[all]{xy}
\usepackage{amsthm}
\usepackage{graphicx}
\usepackage[T1]{fontenc}
\usepackage{hyperref}
\usepackage{caption}
\usepackage{verbatim}
\usepackage{amsmath}
\usepackage{amssymb}
\usepackage{amsthm}
\usepackage{amscd}
\usepackage{graphics}
\usepackage{graphpap}

\newtheorem{definicion}{Definition}[section]
\newtheorem{nota}[definicion]{Remark}
\newtheorem{prop}[definicion]{Proposition}
\newtheorem{lema}[definicion]{Lemma}

\newtheorem{teorema}[definicion]{Theorem}
\newtheorem{cor}[definicion]{Corollary}

\newcommand{\co}{\ensuremath{\colon}} 
\newcommand{\bz}{\ensuremath{\mathbb{Z}}} 
\newcommand{\br}{\ensuremath{\mathbb{R}}} 

\begin{document}

\begin{abstract} Given a group acting on a graph quasi-isometric to a tree, we give sufficent conditions for a pseudocharacter to be bushy. We relate this with the conditions studied by M. Bestvina and K. Fujiwara on their work on bounded cohomology and obtain some results on the space of pseudocharacters.
\end{abstract}

\maketitle
\tableofcontents

\section{Introduction}

If $G$ is a finitely presented group, then $f\co G\to \br$ is a \emph{quasi-homomorphism} or \emph{quasicharacter} if $f(\alpha)+f(\beta)-f(\alpha\beta)$ is bounded independent of $\alpha, \beta$.

If $G$ is a finitely presented group, then $f\co G\to \br$ is a \emph{pseudocharacter} if it has the following properties:

\begin{itemize}
\item $f(\alpha^n)=n\alpha$ for all $\alpha\in G$, $n\in \bz$.
\item $\delta f(\alpha,\beta)=f(\alpha)+f(\beta)-f(\alpha\beta)$ is bounded independent of $\alpha$, $\beta$.
\end{itemize}

Cleraly the constant map $f(G)=0$ is a trivial pseudocharacter. 

\begin{nota}\label{nota: pseudocharacter} Note that if $f$ is a quasicharacter and $\phi$ is given by $$\phi(g)=\lim_{n\to \infty}f(g^n)/n,$$ then $\phi$ is a pseudocharcter with $\phi-f$ bounded.
\end{nota}

Let $S$ be a finite generating set for $G$. If $\Gamma(G,S)$ is the Cayley graph associated to the generating set $S$, then $f$ can be extended affinely over the edges of $\Gamma(G,S)$. 

If $\phi\co \br_+ \to \Gamma(G,S)$ is an infinite ray, then the \emph{sign} of $\phi$ is 
$$\sigma(\phi)= \left\{ \begin{tabular}{l} $+1$ if $\lim_{t\to \infty} f\circ \phi(t)=\infty $ \\
$-1$ if $\lim_{t\to \infty} f\circ \phi(t)=-\infty $ \\
\ \ 0 otherwise
\end{tabular}
 \right.$$

If $w$ is some infinite word in the generators $S$, there is a path $\phi_w\co \br_+ \to \Gamma(G,S)$ beginning at 1 and realizing the word. Define $\sigma(w)=\sigma(\phi_w)$. If $g$ is a group element, let $\sigma(g)$ be the sign of $f(g)$. Notice that if we pick a word $w$ representing $g$ then $\sigma(www \cdots)=\sigma(w^\infty)=\sigma(g)$.

\begin{definicion} \[E(f,S)=\Big\{ w=w_1w_2... \, | \, w_i\in S\cup S^{-1} \mbox{ and } \sigma(w)\in \{+1,-1\}\Big\}/\sim. \]
$w=w_1w_2... \sim_C v=v_1,v_2...$ if $\sigma(w)=\sigma(v)$ and for all $D$ with $\sigma(w)D>C$  there is a word $d=d_1...d_n$ in the letters $S\cup S^{-1}$ such that:
\begin{itemize}
\item $w_pd=v_p$ in $G$ for some prefix $w_p$ of $w$ and some prefix $v_p$ of $v$,
\item $|f(w_pd_p)-D|\leq C$ for all prefixes $d_p$ of $d$.
\end{itemize}
The word $d$ will be referred to as a \emph{connecting word} and $w\sim v$ if $w\sim_C v$ for some $C$. This is an equivalence relation.
\end{definicion}

Since the set $E(f,S)$ is invariant under chage of generators, see \cite[2.3]{Man}, it can be denoted just by $E(f)$.

Let $f\co G \to \br$ be a pseudocharacter. $E(f)^+\subset E(f)$ denotes the set of positive elements of $E(f)$, and $E(f)^-\subset E(f)$ the set of negative elements. If $|E(f)|=2$, $f$ is said to be \emph{uniform}. If $|E(f)^+|=1$ or $|E(f)^-|=1$ but $f$ is not uniform, then $f$ is said to be \emph{unipotent}. Otherwise, $f$ is said to be \emph{bushy}.

This work is mainly based in two papers. The first one is due to M. Bestvina and K. Fujiwara, \cite{BF}. In the first part of that work they consider a group acting on a $\delta$-hyperbolic graph by isometries. There, they conclude a sequence of papers from the second author, see \cite{EF}, \cite{F1} and \cite{F2}, proving that if the action holds certain conditions (J. Manning called this a \emph{Bestvina-Fujiwara action}), then the dimension of the second bounded cohomology of $G$ as a vector space over $\br$ is the cardinal of the continuum.  

On the other hand, J. Manning proves in \cite{Man} two interesting results about pseudocharacters. In the first one the author proves that if for a given group $G$ there is a non-uniform pseudocharacter, then $G$ admits a cobounded quasi-action on a bushy tree. To do that, he also defines \emph{Bottleneck Property} characterizing when a metric space is quasi-isometric to a tree.

The second one relates the existence of a \emph{bushy} pseudocharacter with the conditions on the action studied in \cite{BF}:

\begin{prop}\cite[4.27]{Man} If $f\co G \to \br$ is a bushy pseudocharacter, then there is a Bestvina-Fujiwara action of $G$ on a quasi-tree.
\end{prop}

Herein we work in the oposite direction. In section 2, we give some sufficent conditions for the existence of non-uniform pseudocharacters. 

\begin{prop}[\ref{Prop: bornologous}]\label{Prop: bornologous-introd} Let $G$ be a group acting on a quasi-tree $X$. Let $g_1,g_2$ be two hyperbolic elements of $G$ such that $d_H(\Gamma_1(g_1,x_0,\gamma_1),\Gamma_2(g_2,x_0,\gamma_2))=\infty$. Then, if $h$ is a pseudocharacter such that 
$h(g_1)>0$ and $h(g_2)>0$ and $h$ is bornologous on the action
then, $g_1^\infty\not\sim g_2^\infty$ in $E(h)$.
\end{prop}

\begin{cor}[\ref{Cor: metrically proper}]\label{Cor: metrically proper-introd} Let $G$ be a group acting by isometries on a quasi-tree $X$ so that the action is metrically proper. Let $g_1,g_2$ be two hyperbolic elements of $G$ such that $d_H(\Gamma_1(g_1,x_0,\gamma_1),\Gamma_2(g_2,x_0,\gamma_2))=\infty$. Then, if $h$ is a pseudocharacter such that 
$h(g_1)>0$ and $h(g_2)>0$ 
then, $g_1^\infty\not\sim g_2^\infty$ in $E(h)$.
\end{cor}

\begin{cor}[\ref{Cor: nonelementary}] Consider a nonelementary action of a group $G$ on a quasi-tree $X$. If the action is metrically proper then every nonelementary pseudocharacter is bushy. 
\end{cor}

In section 3, we prove that given a Bestvina-Fujiwara action, it holds the conditions of Proposition \ref{Prop: bornologous}. Moreover, 

\begin{teorema}[\ref{Tma: bushy}]\label{Tma: bushy-introd} Let $G$ be a group acting on a quasi-tree. If it is a Bestvina-Fujiwara action, then there is a bushy pseudocharacter $h\co G \to \br$. 
\end{teorema}

A quasi-action of a group $G$ on a metric space $X$ associates to each $g\in G$ a quasi-isometry $A_g\co X \to X$ with uniform quasi-isometry constants so that $A_{Id}=Id_X$ and such that the distance between $A_h\circ A_g$ and $A_{hg}$ in the sup norm is uniformly bounded independent of $g,h\in G$. This is a natural and interesting extension of group actions and it has been relevant in relation to trees. In \cite{MSW}, the authors prove that every cobounded quasi-action on a bounded valence bushy tree is quasi-conjugate to an action on a tree. However, there are examples of quasi-actions on simplicial trees which are not quasi-conjugate to actions on \br-trees. See \cite{Man} for the examples and \cite{Man2} for further results on quasi-actions on trees. 

Given a non-uniform pseudocharacter $h$, Manning builds in \cite{Man} a cobounded quasi-action on a bushy tree $T$. In section 4 we show that this construction can be made adding certain condition in the relation between the space $E(h)$ and the boundary of the tree $\partial T$.

In the last section we state some implications on the space of pseudocharacters and, therefore, in the cobounded cohomology of the group. 

\begin{cor}[\ref{Cor: bushy}] If there is a bushy pseudocharacter $h\co G \to \br$ then, the dimension of the space generated by the bushy pseudocharacters on $G$ as a vector space over $\br$ is the cardinal of the continuum. 
\end{cor}

All groups are assumed to be finitely presented.

\section{Actions and pseudocharacters}

\begin{definicion} A map between metric spaces, $f:(X,d_X)\to (Y,d_Y)$, is
said to be \emph{quasi-isometric} if there are constants $\lambda
\geq 1$ and $C>0$ such that $\forall x,x'\in X$,
$\frac{1}{\lambda}d_X(x,x') -A \leq d_Y(f(x),f(x'))\leq \lambda
d_X(x,x')+A$. If there is a constant $B>0$ such that
$N_B(f(X))=Y$ where $N_B(f(X))=\{y\in Y\, | \, d_Y(y,f(X))< B\}$, then $f$ is a
\emph{quasi-isometry} and $X,Y$ are \emph{quasi-isometric}.
\end{definicion}

\begin{teorema}\cite[Theorem 4.6]{Man}\label{Bottleneck} Let $Y$ be a geodesic metric space. The following are equivalent:
\begin{itemize}\item[(1)] $Y$ is quasi-isometric to some simplicial tree $\Gamma$.
\item[(2)] (Bottleneck Property) There is some $\Delta>0$ so that for all $x,y \in Y$ there is a midpoint $m=m(x,y)$ with $d(x,m)=d(y,m)=\frac{1}{2}d(x,y)$ and the property that any path from $x$ to $y$ must pass within less than $\Delta$ of the point $m$.
\end{itemize}
\end{teorema}

Let $(X,d)$ be a metric space. Fix a base point $o\in X$ and for
$x,x'\in X$ put $(x|x')_o=\frac{1}{2}(d(x,o)+d(x',o)-d(x,x'))$. The number $(x|x')_o$
is non-negative and it is called the \emph{Gromov product} of
$x,x'$ with respect to $o$. See \cite{Gr}.

\begin{definicion} A metric space $X$ is \emph{(Gromov)
hyperbolic} if it satisfies the $\delta$-inequality
\[(x|y)_o\geq min\{(x|z)_o,(z|y)_o\}-\delta\] for some $\delta\geq
0$, for every base point $o\in X$ and all $x,y,z \in X$.
\end{definicion}

Let $X$ be a hyperbolic space and $o\in X$ a base point. A
sequence of points $\{x_i\}\subset X$ \emph{converges to infinity}
if \[\lim_{i,j\to \infty} (x_i|x_j)_o=\infty.\] This property is
independent of the choice of $o$ since
\[|(x|x')_o-(x|x')_{o'}|\leq |oo'|\] for any $x,x',o,o' \in X$.

Two sequences $\{x_i\},\{x'_i\}$ that converge to infinity are
\emph{equivalent} if \[\lim_{i\to \infty} (x_i|x'_i)_o=\infty.\]
Using the $\delta$-inequality, we easily see that this defines an
equivalence relation for sequences in $X$ converging to infinity.
The \emph{boundary at infinity} $\partial_\infty X$ of $X$ is
defined to be the set of equivalence classes of sequences
converging to infinity.

The following lemma is a well known property of quasi-geodesics (see \cite{Bow} or \cite{F1}). The statement with the proof can be found in \cite{Man2}.

\begin{lema}\label{Lemma: stab} Let $X$ be a $\delta$-hyperbolic space. Given $K\geq 1$ and $L \geq 0$, there exists $B(K,L,\delta)\geq 0$ such that if $\gamma_1, \gamma_2$ are two $(K,L)$-quasi-geodesics with the same endpoints in $X\cup \partial X$, then $\gamma_1\subset N_B(\gamma_2)$ and $\gamma_2\subset N_B(\gamma_1)$. 
\end{lema}

\begin{definicion} Fix $x_0\in X$, where $X$ is a $\delta$-hyperbolic metric space on which $G$ quasi-acts. Let $O_{g,x}\co \br \to X$ be defined by $O_{g,x}(t)=g^{\left\lfloor t\right\rfloor}x$ where $\left\lfloor t\right\rfloor$ is the largest integer smaller than $t$.  Then it is said that $g$ \emph{quasi-acts elliptically} if $O_{g,x}$ has bounded image, and $g$ \emph{quasi-acts hyperbolically} if $O_{g,x}$ is a quasi-geodesic. If $G$ acts isometrically on $X$ then it is said that $g$ \emph{acts} elliptically or hyperbolically or that $g$ is \emph{elliptic} or \emph{hyperbolic}.
\end{definicion}

It is readily seen that this definition is independent of $x$ and agrees with the standard definitions in case $G$ acts isometrically.

If $g\in G$ is hyperbolic $x\in X$, and $\gamma_0\co [0,1]\to X$ is a geodesic segment with $\gamma_0(0)=x$ and $\gamma_0(1)=g(x)$, then it is not hard to check that $\Gamma_{g,x,\gamma_0} \co \br \to X$ given by \begin{equation}\Gamma_{g,x,\gamma_0}(t)=g^{\left\lfloor t\right\rfloor}\gamma_0(t-\left\lfloor t\right\rfloor) \label{quasi-geodesic} \end{equation} is a continuous quasi-geodesic. Moreover, $g$ is an isometry of $X$ which maps this quasi-geodesic to itself by a nontrivial translation. See Figure \ref{fig: Gamma}.

\begin{figure}
\centering 
\includegraphics[scale=0.5]{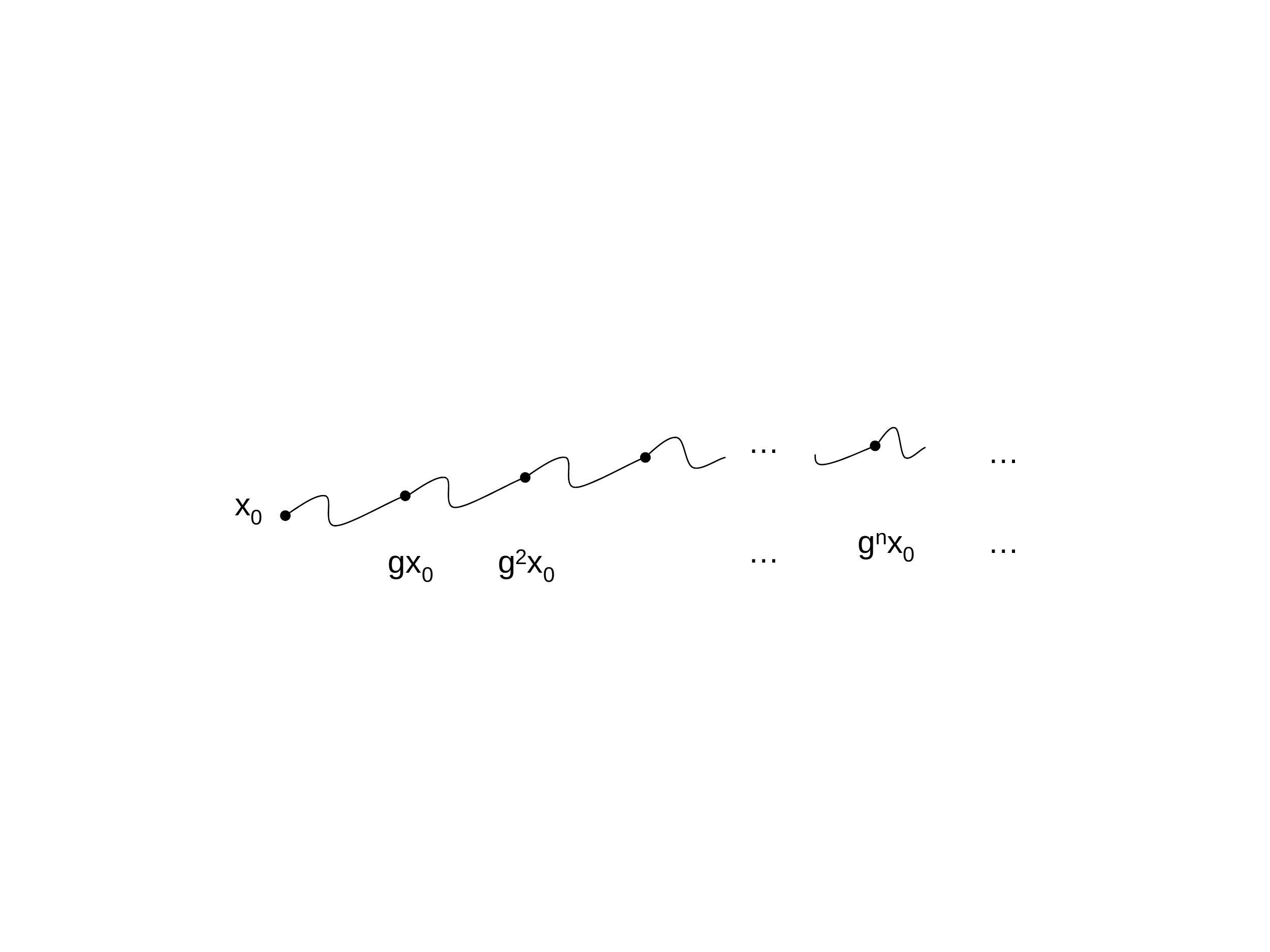}
\caption{$g$ acts on $\Gamma_{g,x,\gamma_0}$ by a nontrivial translation.}
\label{fig: Gamma}
\end{figure}


A quasi-geodesic where $g$ acts by nontrivial translation will be referred to as a \emph{quasi-axis} (or $(K-L)$-\emph{quasi-axis} if the constants are relevant). A quasi-axis of $g$ is given the $g$-orientation by the requirement that $g$ acts as a positive translation.

\begin{definicion} Let $G$ be a group acting by isometries on a metric space $X$. We say that a pseudocharacter $h\co G \to \br$ is \textbf{bornologous on the action} if given any $x_0\in X$ and any $g\in G$, $\forall R>0$ there exists $S>0$ such that $\forall g'\in G$ with $g'(x_0)\in B(g(x_0),R)$,  $|h(g')-h(g)|\leq S$. 
\end{definicion}

A \emph{quasi-tree} is a complete geodesic metric space quasi-isometric to some simplicial tree. These spaces satisfy bottleneck property, see \ref{Bottleneck}. Herein, we will add the asumption that the quasi-tree is a graph. This is not a restrictive asumption since we are working in a coarse setting but it has obvious technical advantadges. Therefore, from now on, a \textbf{quasi-tree} will be a graph satisfying bottleneck property.

\begin{prop}\label{Prop: bornologous} Let $G$ be a group acting on a quasi-tree $X$. Let $g_1,g_2$ be two hyperbolic elements of $G$ such that $d_H(\Gamma_1(g_1,x_0,\gamma_1),\Gamma_2(g_2,x_0,\gamma_2))=\infty$. Then, if $h$ is a pseudocharacter such that 
$h(g_1)>0$ and $h(g_2)>0$ and $h$ is bornologous on the action
then, $g_1^\infty\not\sim g_2^\infty$ in $E(h)$.
\end{prop}

\begin{proof} Since $h(g_1)>0$ and $h(g_2)>0$, $\sigma_h(g_1)> 0$ and $\sigma_h(g_2)> 0$. 
Let $w_j$ be the word representing $g_j$ in the letters of the generating set $S\cup S^{-1}$. Then $w_j w_j ...=w_j^\infty$ is an element of $E(h)$ fixed by $g_j$ for $j=1,2$. Note that, with the assumptions taken, $\sigma(w_j^\infty)=+1$.

Let us see that $w_1^\infty\not \sim w_2^\infty$ in $E(h)$. 

Let us denote, for simplicity, $w=w_1^\infty$ and $v=w_2^\infty$ and suppose $w \sim_C v$ for some $C>0$. Then, given $D_1>C$  there is a connecting word $d=d_1...d_n$ such that:
\begin{itemize}
\item $w_pd=v_p$ in $G$ for some prefix $w_p$ of $w$ and some prefix $v_p$ of $v$,
\item $|h(w_pd_p)-D_1|\leq C$ for all prefixes $d_p$ of $d$.
\end{itemize}

By abuse of notation let us identify the group element $g$ with the word representing it, $w$. Therefore, we write $w(x_0)$ for the image of the isometric action $g$ on $x_0$. Let $\gamma$ be a geodesic path from $w_p(x_0)$ to $v_p(x_0)$ and let $m$ be the midpoint in $\gamma$. 
 
Let $\gamma_1$ be a geodesic path from $x_0$ to $w_1(x_0)$ and $\gamma_2$ be a geodesic path from $x_0$ to $w_2(x_0)$ and consider $\Gamma_1:=\Gamma_{g_1,x_0,\gamma_1}(t)$, $\Gamma_2:=\Gamma_{g_2,x_0,\gamma_2}(t)$ two continuous quasi-geodesics defined as in (\ref{quasi-geodesic}). Let $\Gamma_1(w_p,w_q)$ be the restriction of $\Gamma_{g_1,x_0,\gamma_1}(t)$ to a (quasi-isometric) path from $w_p(x_0)$ to $w_q(x_0)$ for any prefixes $w_p,w_q$ of $w$. Also, let $\Gamma_2(v_p,v_q)$ be the restriction of $\Gamma_{g_2,x_0,\gamma_2}(t)$ to a (quasi-isometric) path from $v_p(x_0)$ to $v_q(x_0)$ for any prefixes $v_p,v_q$ of $v$.

Let $\Delta$ be the bottelneck property constant.

Claim. Since $d_H(\Gamma_1(g_1,x_0,\gamma_1),\Gamma_2(g_2,x_0,\gamma_2))=\infty$, we may assume $D_1$ big enough to assure that for any $w_p\subset w_q$ and $v_p\subset v_q$,  then $\Gamma_1(w_p,w_q)\cap B(m,\Delta)=\emptyset$ and $\Gamma_2(v_p,v_q)\cap B(m,\Delta)=\emptyset$. See Figure \ref{fig: D1}.

Since $d_H(\Gamma_1(g_1,x_0,\gamma_1),\Gamma_2(g_2,x_0,\gamma_2))=\infty$, we may assume $D_1=D_1(h,\Delta, g_1,g_2)$ big enough to assure that $A:=d(w_p,v_p)$ is as big as we want. Notice that $d(w_1^i(x_0),w_1^{i+1}(x_0))= d(x_0,w_1(x_0))=:d_1$ and hence $d(w_1^i(x_0),w_1^{i+k}(x_0))\leq  k\cdot d_1$. Respectively, for $w_2$, $d(w_2^i(x_0),w_2^{i+1}(x_0))= d(x_0,w_2(x_0))=:d_2$ and hence $d(w_2^i(x_0),w_2^{i+k}(x_0))\leq  k\cdot d_1$. So, if $A$ is big enough, either $k$ is also big enough (depending on the distance $(A/2-\Delta)/\max\{d_1,d_2\}$) or we can assure that the quasi-geodesic $\Gamma_1(w_p,w_p\cdot w_1^k)$ from $w_p(x_0)$ to $w_p\cdot w_1^k(x_0)$ holds that $\Gamma_1(w_p,w_p\cdot w_1^k)\cap B(m,\Delta)=\emptyset$. Also, either $k$ is big enough or the quasi-geodesic $\Gamma_2(v_p,v_p\cdot w_2^k)$ from $v_p(x_0)$ to $v_p\cdot w_2^k(x_0)$ holds that $\Gamma_2(v_p,v_p\cdot w_2^k)\cap B(m,\Delta)=\emptyset$. 

Now, let us assume $k=k(h,\Delta, g_1,g_2)$ as big as we want and fix it asuming that the corresponding quasi-geodesics $\Gamma_1(w_p,w_p\cdot w_1^k)$ and $\Gamma_2(v_p,v_p\cdot w_2^k)$ do not intersect  the ball $B(m,\Delta)$. By hypothesis $h(g_j)>0$ and $h(g_j^n)=nh(g_j)$ for $j=1,2$. Then, for any $j\geq k$, $h(w_p\cdot w_1^j(x_0))=h(w_p)+jh(g_1)$ and $h(v_p\cdot w_2^j)=h(v_p)+jh(g_2)$ are much bigger than $D_1$. 
Then, since $h$ is bornologous on the action, we can assume $k$ big enough so that $w_p\cdot w_1^j(x_0)\not\in B(m,\Delta + \max\{d_1,d_2\})$ and $v_p\cdot w_2^j(x_0)\not\in B(m,\Delta +  \max\{d_1,d_2\})$ for any $j\geq k$. Thus, the quasi-geodesic segments $\Gamma_1(w_p\cdot w_1^j,w_p\cdot w_1^{j+1})\cap B(m,\Delta)=\emptyset$ and 
$\Gamma_2(v_p\cdot w_2^j,v_p\cdot w_2^{j+1})\cap B(m,\Delta)=\emptyset$ for every $j\geq 0$ proving the claim. \hfill$\square$

\begin{figure}
\centering 
\includegraphics[scale=0.4]{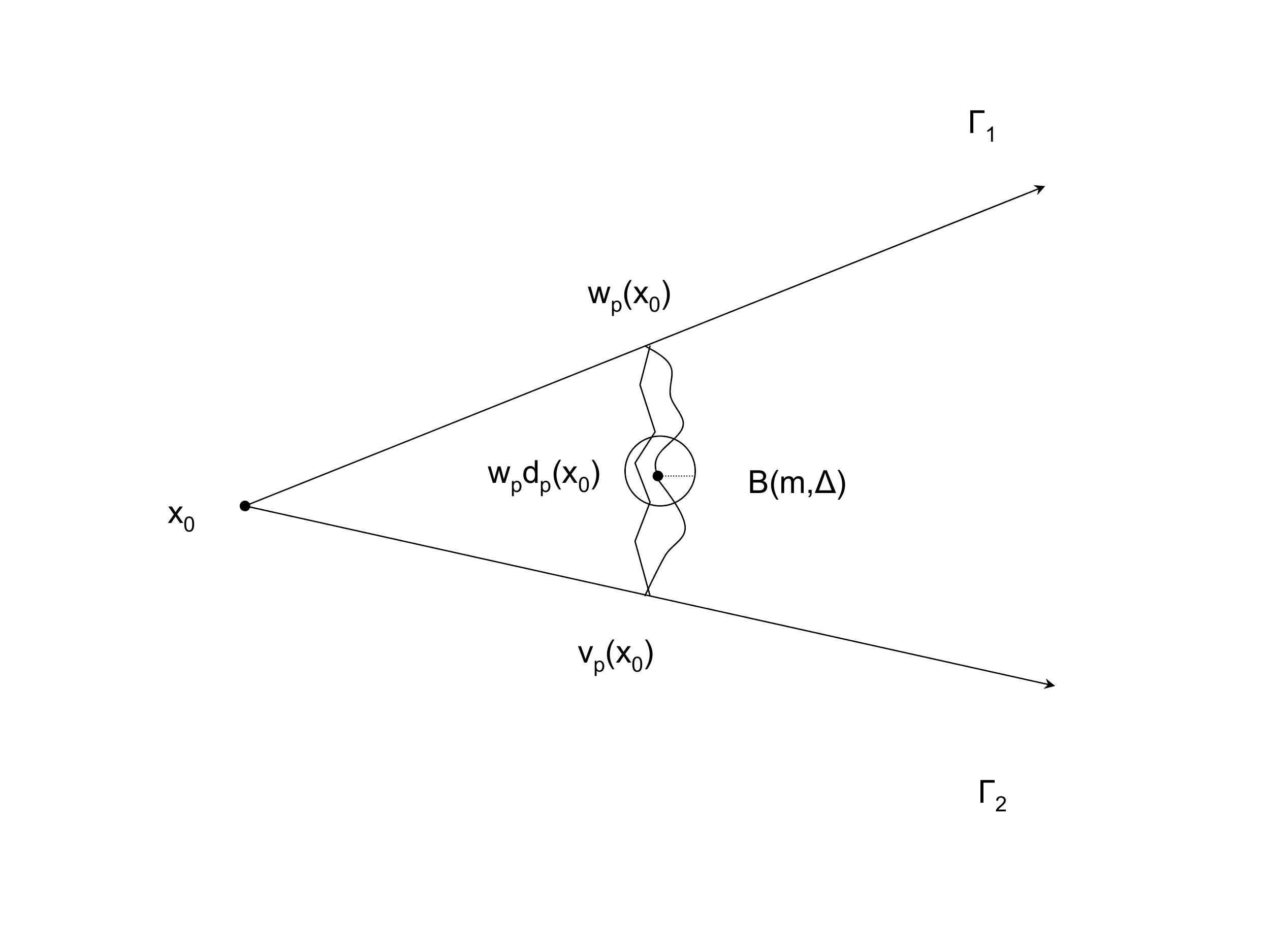}
\caption{Assuming $D_1$ big enough, we obtain that $\Gamma_1(w_p,w_p\cdot w_1^k)$ and $\Gamma_2(v_p,v_p\cdot w_2^k)$ do not intersect  the ball $B(m,\Delta)$.}
\label{fig: D1}
\end{figure}

Now, let $D_2>>h(m),D_1$. Then, we will reach a contradiction finding a uniform bound for $D_2-D_1$. 

By asumption, there is a connecting word $d'=d'_1...d'_n$ such that:
\begin{itemize}
\item $w_{p'}d'=v_{p'}$ in $G$ for some prefix $w_{p'}$ of $w$ and some prefix $v_{p'}$ of $v$,
\item $|h(w_{p'}d'_{p'})-D_2|\leq C$ for all prefixes $d'_{p'}$ of $d$.
\end{itemize}

Let $\Gamma_1:=\Gamma_1(w_p,w_p')$, $\Gamma_2:=\Gamma_2(v_p,v_p')$ quasi-geodesic paths defined as above from $w_p(x_0)$ to $w_{p'}(x_0)$ and from $v_p(x_0)$ to $v_{p'}(x_0)$. Let $\gamma'_j\co [0,1]$ be a geodesic path from $w_{p'}\cdot d'_{p'}\cdot d'_{j-1}(x_0)$ to $w_{p'}\cdot d'_{p'}\cdot d'_{j}(x_0)$ for $1\leq j \leq n$ and let $\gamma'\co [0,n]\to X$ be the path from $w_{p'}(x_0)$ to $v_{p'}(x_0)$ defined by $d'$ where $\gamma'(t)=\gamma'_{\left\lfloor t\right\rfloor}(t-\left\lfloor t\right\rfloor)$.

Then, $\Gamma=\Gamma_1\cup \Gamma' \cup \Gamma_2^{-1}$ is a path from $w_p(x_0)$ to $v_p(x_0)$. By bottleneck property, see \ref{Bottleneck}, there is a point $x\in \Gamma$ such that $d(x,m)\leq \Delta$. By the previous claim, we can assume that $x\in \gamma'_j[0,1]$ for some $1\leq j\leq n$. See Figure \ref{fig: D2}.

\begin{figure}
\centering 
\includegraphics[scale=0.4]{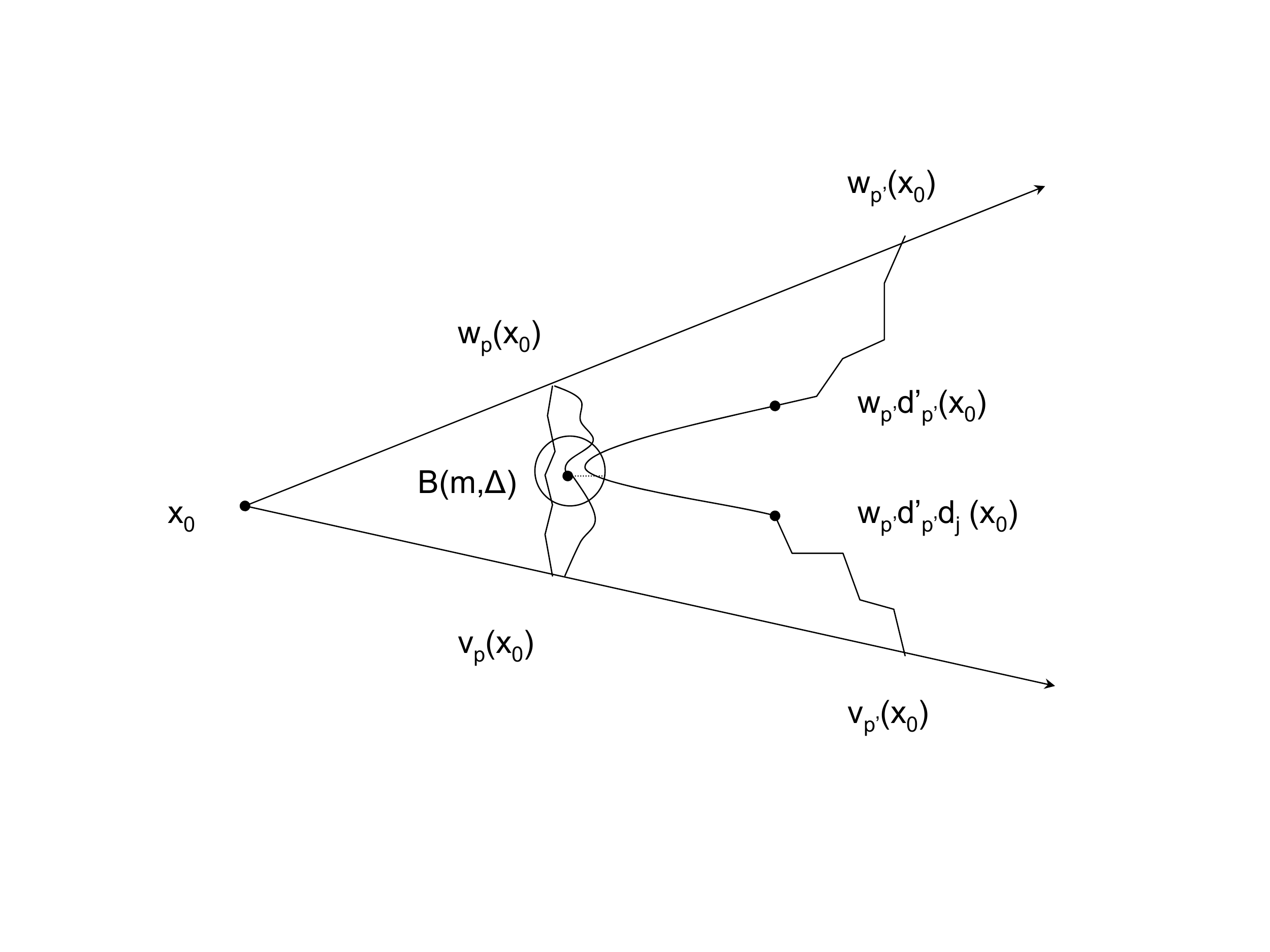}
\caption{By bottleneck property the path $[w_p(x_0),w_{p'}(x_0)]\cup [w_{p'}(x_0),v_{p'}(x_0)]\cup [v_{p'}(x_0),v_{p}(x_0)]$ intersects the ball $B(m,\Delta)$.}
\label{fig: D2}
\end{figure}

By hypothesis, $h(w_{p'}d'_{p'}),h(w_{p'}d'_{p'}d'_j)\in (D_2-C,D_2+C)$. Let $F=F(S,x_0)$ be a constant such that $\forall s\in S$, $d(x_0,s(x_0))\leq F$. Then, $d(w_{p'}d'_{p'}(x_0),w_{p'}d'_{p'}d'_j(x_0))\leq F$. Therefore, $d(w_{p'}d'_{p'}(x_0),m)\leq  \Delta +F$ and $d(w_pd'_pd'_j(x_0),m)\leq  \Delta +F $ which implies that 
$d(w_{p'}d'_{p'}(x_0),w_pd_p(x_0))\leq  2\Delta +2F$ and $d(w_pd'_pd'_j(x_0),w_pd_p(x_0))\leq  2\Delta +2F$ for some prefix $d_p$ of $d$. But since $h$ is bornologous on the action, there is some $S=S(2\Delta+2F)$ such that $|h(w_pd_p)-h(w_pd'_pd'_j)|<S$ and $D_2-D_1<S+2C$.

Thus, $D_2-D_1$ is bounded by a constant depending only on $h,C,F,\gamma_1,\gamma_2$ and $\Delta$ leading to contradiction.
\end{proof}

\begin{definicion} An action of a group by isometries on a metric space is \emph{metrically proper} if $\forall x\in X$ and $\forall R>0$ the set $\{g\in G\, | \, g(N(x,R))\cap N(x,R)\neq \emptyset\}$ is finite.
\end{definicion}

\begin{lema} An action of a group $G$ by isometries on a metric space $X$ is metrically proper if and only if $\forall x_0\in X$ and $\forall g\in G$ then $\forall R>0$ the set $\{g'\in G\, | \, g'x_0\in N(g(x_0),R)\}$ is finite.
\end{lema}

\begin{proof} If the action is metrically proper, the set $\{h\in G\, | \, h(B(g(x_0),R))\cap B(g(x_0),R)\neq \emptyset\}$ is finite. If we consider $g'=hg$ then the set $\{g'\in G\, | \, g'g^{-1}B(g(x_0),R)\cap B(g(x_0),R)\neq \emptyset\}=\{g'\in G\, | \, g'B(x_0,R)\cap B(g(x_0),R)\neq \emptyset\}$ is finite. In particular, the set $\{g'\in G\, | \, g'x_0\in B(g(x_0),R)\}$ is finite.

Conversely, suppose that the set $\{g'\in G\, | \, g'x_0\in N(g(x_0),2R)\}$ is finite. Then, considering $h=g'g^{-1}$, $\{h\in G\, | \, (hg)(x_0)\cap N(g(x_0),2R))\neq \emptyset\}=\{h\in G\, | \, h(g(x_0))\cap N(g(x_0),2R))\neq \emptyset\}$ is finite which implies that the set $\{h\in G\, | \, h(N(g(x_0),R))\cap N(g(x_0),R))\neq \emptyset\}$ is finite.
\end{proof}

\begin{prop} Let $G$ be a group acting by isometries on a metric space $X$ and let $h\co G \to \br$ be a pseudocharacter. If the action is metrically proper then, the pseudocharacter is bornologous on the action.
\end{prop}

\begin{proof} Let $x_0\in X$, $g\in G$ and $R>0$. Since the action is metrically proper the set $K=\{g'\in G\, | \, g'x_0\in N(g(x_0),R)\}$ is finite. Therefore, it suffices to take $S:=\max_{g'\in K}\{|h(g')-h(g)|\}$.
\end{proof}

\begin{cor}\label{Cor: metrically proper} Let $G$ be a group acting by isometries on a quasi-tree $X$ so that the action is metrically proper. Let $g_1,g_2$ be two hyperbolic elements of $G$ such that $d_H(\Gamma_1(g_1,x_0,\gamma_1),\Gamma_2(g_2,x_0,\gamma_2))=\infty$. Then, if $h$ is a pseudocharacter such that 
$h(g_1)>0$ and $h(g_2)>0$ 
then, $g_1^\infty\not\sim g_2^\infty$ in $E(h)$.
\end{cor}

\begin{definicion} Two hyperbolic isometries $g_1,g_2$ are said to be \emph{independent} if their quasi-axis do not contain rays which are a finite Hausdorff distance apart. Equivalently the fixed point sets of $g_1,g_2$ in $\partial X$ are disjoint. An action is \emph{nonelementary} if there are at least two independent hyperbolic elements.
\end{definicion}

\begin{definicion} We say that a pseudocharacter $h\co G \to \br$ is \textbf{nonelementary} if there is a pair of independent $g_1,g_2\in G$ such that $h(g_1)\neq 0$ and $h(g_2)\neq 0$. 
\end{definicion}

\begin{cor}\label{Cor: nonelementary0} Consider a nonelementary action of a group $G$ on a quasi-tree $X$ and a nonelementary pseudocharacter $h\co G \to \br$. Then, if $h$ is bornologous on the action, it is bushy. 
\end{cor}

\begin{cor}\label{Cor: nonelementary} Consider a nonelementary action of a group $G$ on a quasi-tree $X$. If the action is metrically proper then every nonelementary pseudocharacter is bushy. 
\end{cor}

\section{Existence of bushy pseudocharacters}

Note that any two $(K-L)$-\emph{quasi-axis} of $g$ are within some universal $B=B(\delta,K,L)$ of one another and any sufficiently long $(K,L)$-quasi-geodesic arc $J$ in a $B$-neighborhood of a quasi-axis $l$ of $g$ inherits a natural $g$-orientation: a  point of $l$ within $B(\delta,K,L)$ of the terminal endpoint of $J$ is ahead (with trepect to the $g$-orientation of $l$) of a point of $l$ within $B(\delta,K,L)$ of the initial endpoint of $J$. 
This orientation of $J$ will be denoted by $g$-orientation of $J$.

\begin{definicion}\cite{BF}\label{def: sim} If  $g_1,g_2$ are hyperbolic elements of $G$, let $g_1\sim g_2$ if for an arbitrarily long segment $J$ in a $(K,L)$-quasi-axis for $g_1$ there is a $g\in G$ such that $g(J)$ is within $B(\delta,K,L)$ of a $(K,L)$-quasi-axis of $g_2$ and $g\co J\to g(J)$ is orientation-preserving with respect to the $g_2$-orientation on $g(J)$.
\end{definicion}

This defines an equivalence relation. As it is said in \cite{BF}, the concept does not change if $B$ is replaced by a larger constant. 

\begin{definicion} A \emph{Bestvina-Fujiwara} action is a nonelementary action of a group $G$ on a hyperbolic graph $X$ so that there exist independent $g_1, g_2$ such that $g_1\not \sim g_2$.
\end{definicion}

\begin{lema} \label{lema: divergent} Let $X$ be a (geodesic) Gromov hyperbolic space and let $G$ be a group acting by isometries on $X$. Let $g_1,g_2\in G$ such that $g_1\not \sim g_2$. Then, for any point $x_0\in X$, and any pair of geodesics $\gamma_1:=[x_0,g_1(x_0)],\gamma_2:=[x_0,g_2(x_0)]$, $d_H(\Gamma_1(g_1,x_0,\gamma_1),\Gamma_2(g_2,x_0,\gamma_2))=\infty$.
\end{lema}

\begin{proof} Suppose there is some $B>0$ such that $\Gamma_1\subset N_B(\Gamma_2)$ and $\Gamma_2\subset N_B(\Gamma_1)$. Then, the action of the identity $i$ on $\Gamma_1$ is contained in $N_B(\Gamma_2)$ and $i\co \Gamma_1 \to i(\Gamma_1)=\Gamma_1$ is obviously orientation preserving with respect to the $g_1$-orientation and the $g_2$-orientation on $\Gamma_1$. This contradicts the fact that $g_1\not \sim g_2$. 
\end{proof}

Let us recall now the basic construction of quasi-homomorphisms associated to the action as presented in \cite{F1}.

Let $X$ be a hyperbolic graph and a group $G$ acting on $X$.
Let $w$ be a finite (oriented) path in $X$. By $|w|$ denote the lenght of $w$, by $i(w)$ the starting point and by $t(w)$ the finishing point. For $g\in G$, $g\circ w$ is a path starting at $g(i(w))$ and finishing at $g(t(w))$ and it is called a \emph{copy} of $w$. Obviously, $|g\circ w|=|w|$.

Let $\alpha$ be a finite path. Define \[|\alpha|_w=\{\mbox{the maximal number of non-overlapping copies of $w$ in $\alpha$}\}.\]

Suppose that $x,y\in X$ are two vertices and that $W$ is an integer with $0<W<|w|$. Then, $$c_{w,W}(x,y)=d(x,y)-\inf_\alpha(|\alpha|-W|\alpha|_w),$$ where $\alpha$ ranges over all paths from $x$ to $y$. 

\begin{lema}\cite[Lemma 3.4]{F1} \label{lema: 3.4 F1} Let $x,\, y,\, z$ be three points in $X$. Then 
\[|c_{w,W}(x,y)-c_{w,W}(x,z)|\leq 2d(y,z).\]\hfill$\blacksquare$
\end{lema}

Let us omit $W$ from the notation and write $c_w$. Define $h_{w}\co G\to \br$ by 
$$h_{w}(g)=c_{w}(x_0,g(x_0))-c_{w^{-1}}(x_0,g(x_0)).$$ 

\begin{prop}\cite[Proposition 3.10]{F1} If $X$ is a $\delta$-hyperbolic space $h_w$ is a  quasi-homomorphism (i. e. quasicharacter).
\end{prop}

From Lemma \ref{lema: 3.4 F1}, it is immediate to obtain the following.

\begin{lema}\label{lema: bounded} Let $X$ be a (geodesic) Gromov hyperbolic space. For any word $w$, any points $x_0,x\in X$ and any constants $0<W<|w|$, $R>0$, the subset of the real line $\{h_{w}(g) \, \co \, g(x_0)\in B(x,R)\}$ is bounded. In particular, it has diameter at most $8R$. 
\end{lema}

Therefore, the following is immediate.

\begin{prop}\label{prop: bounded} Let $X$ be a (geodesic) Gromov hyperbolic space. For any word $w$, any points $x_0,x\in X$ and any constants $0<W<|w|$, $h_w$ is bornologous on the action.
\end{prop}

Let us recall two propositions from \cite{BF}.

\begin{prop}\cite[Proposition 2]{BF}\label{Prop2: BF} Suppose a group $G$ acts on a $\delta$-hyperbolic graph $X$ by isometries. Suppose also that the action is nonelementary and that there exist independent hyperbolic elements $g_1,g_2\in G$ such that $g_1 \not \sim g_2$. Then, there is a sequence $f_1,f_2,... \in G$ of hyperbolic elements such that 
\begin{itemize}
\item $f_i \not \sim f_i^{-1}$ for $i=1,2,...$ and
\item $f_i \not \sim f_j^{\pm 1}$ for $j<i$.
\end{itemize}
\end{prop}

Replacing if necessary $g_1,g_2$ by high positive powers of conjugates, let $F$ be a free subgroup of $G$ with basis $\{g_1,g_2\}$ such that each nontrivial element of $F$ is hyperbolic and $F$ is quasi-convex with respect to the action on $X$. See the proof of \cite[Propostion 2]{BF} and \cite[Section 5.3]{Gr} for details. Such free groups are called \emph{Schottky groups}.

\begin{prop}\cite[Proposition 5]{BF}\label{Prop5: BF} Suppose $1\neq f \in F$ is cyclically reduced and $f\not \sim f^{-1}$. Then there is $a>0$ such that $h_{f^a}$ is unbounded on $<f>$. Moreover, if $f^{\pm 1}\not \sim f'\in F$ then $h_{f^a}$ is 0 on $<f'>$ for sufficiently large $a>0$. 
\end{prop}

\begin{teorema}\label{Tma: bushy} Let $G$ be a group acting on a quasi-tree. If it is a Bestvina-Fujiwara action, then there is a bushy pseudocharacter $h\co G \to \br$. 
\end{teorema}

\begin{proof} Consider the sequence $f_1,f_1,...$ from \ref{Prop2: BF} and assume in addition (without loss of generality) that each $f_i$ is cyclically reduced. Define $h_i\co G\to \br$ as $h_i=h_{f^{a_i}}$ where $a_i$ is chosen as in Proposition \ref{Prop5: BF} so that $h_i$ is unbounded on $<f_i>$ and so that it is 0 on $<f_j>$ for $j<i$ and also 0 on $<f_{i+1}>$. With the same argument, we may also assume that $lim_{k\to \infty} h_i(f_i^k)=+\infty$.

Let $h$ be a pseudocharacter at a bounded distance (see Remark \ref{nota: pseudocharacter}) from the quasicharacter $h_1+h_2$. (Notice that everything works if we consider $h_{i}$ and $h_{i+1}$ instead). Clearly, $h(f_1)>0$ and $h(f_2)>0$. 

Therefore $\sigma_h(f_j)> 0$ and $f_j^{\pm \infty}$ defines an element 
in $E(h)$ for $j=1,2$. Let $w_j$ be the word representing $f_j$ in the letters $S\cup S^{-1}$. Then 
$w_j w_j ...=w_j^\infty$ is an element of $E(f)$ fixed by $f_j$ for $j=1,2$. Note that, with the assumptions taken, $\sigma(w_j^\infty)=+1$.

Let us see that $w_1^\infty\not \sim w_2^\infty$ in $E(h)$. It suffices to check that by Proposition \ref{prop: bounded} and Lemma \ref{lema: divergent} we are in the conditions of Proposition \ref{Prop: bornologous}.

The same argument proves that $w_1^{-\infty}\not \sim w_2^{-\infty}$ in $E(h)$.
\end{proof}

\begin{cor} If a Cayley graph $X=\Gamma(G,S)$ satisfies the bottleneck property and the canonical action of the group is a Bestvina-Fujiwara action, then there is a bushy pseudocharacter $h\co G \to \br$. 
\end{cor}

\section{Quasi-actions on trees}
 
Given a pseudocharacter $f\co G \to \br$, J. Manning introduces the following constructions. The firs one gives a tree obtained from the Cayley graph of the group.

Consider an (unambiguous) triangular genarting set $S$. Then scale $f$ so that $f(G)$ misses $\bz+\frac{1}{2}$ and so that $f$ changes by at most $\frac{1}{4}$ over each edge. Le $\tilde{K}$ be the simply connected 2-complex obtained by attaching 2-cells according to the relations of the presentation. Then, a tree is built with vertex set in one-to-one correspondence with the components of $\tilde{K}\backslash f^{-1}(\bz+\frac{1}{2})$. The edges correspond to components of $f^{-1}(\bz+\frac{1}{2})$, each of which is some possibly infinite track which separates $\tilde{K}$ into two components. This construction is also the starting point in \cite{M-P} where given a real valued function on a geodesic space we give a sufficent condition for the space to be quasi-isometric to a tree. 

The next appears as Definition 4.9 in \cite{Man}.

Let $V$ be the set of components of $f^{-1}(\bz +\frac{1}{2})$. Then $V$ is in one-to-one correspondence with the set of edges of $T$. Let $X$ be the simplicial graph with vertex set equal to $G\times V$ and the following edge condition: Two distinct vertices $(g,\tau)$ and $(g',\tau')$ are to be connected by an edge if there is some $h$ so that $hg(\tau)$ and $hg'(\tau')$ are contained in the same connected component of $f^{-1}[n-\frac{3}{2},n+\frac{1}{2}]$ for some $n$. The zero-skeleton $X^0$ is endowed with a $G$-action by setting $g(g_0,\tau_0)=(gg_0,\tau_0)$. Since this action repects the edge conditin on pairs of vertices, it extends to an action on $X$. We will refer to this particular one as \textbf{Manning's action}.

\begin{prop} \cite[Proposition 4.27]{Man} \label{Prop: 4.27} If $f\co G \to \br$ is a bushy pseudocharacter, then Manning's action is a Bestvina-Fujiwara action.
\end{prop}

\begin{teorema}\cite[Theorem 4.15]{Man} The space $X$ satisfies the Bottlenek Property.
\end{teorema}

Therefore, from Theorem \ref{Tma: bushy}, we can give the following corollary which would be some kind of converse to Proposition \ref{Prop: 4.27}.

\begin{cor} If Manning's action is a Bestvina-Fujiwara action, then there is a bushy pseudocharacter $h\co G \to \br$.
\end{cor} 

\begin{lema}\cite[4.17]{Man} There is an injective map from $E(f)$ to $\partial X$
\end{lema}

\begin{teorema}\cite[4.20]{Man} If $f\co G \to \br$ is a pseudocharacter which is not uniform, then $G$ admits a cobounded quasi-action on a bushy tree.
\end{teorema}

Then, it is readily seen, from the construction of the action and the bushy tree, that Corollary \ref{Cor: metrically proper} yields the following.

\begin{cor} Consider a nonelementary action of a group $G$ on a quasi-tree $X$. If the action is metrically proper then for any pseudocharacter $h\co G \to \br$ and any pair of independent $g_1,g_2\in G$ such that $h(g_1)> 0$ and $h(g_2)> 0$ there is a cobounded quasi-action of $G$ on a bushy tree $T$ such that there is an injective map from $E(h)$ to $\partial T$.
\end{cor}

\section{Space of pseudocharacters}

Given a group $G$, quasi-characters and pseudocharacters are major tools in the study of the bounded cohomology group $H_b^2(G;\br)$ as we can see in \cite{BF}.

The bounded cohomology group $H_b^*(G;\br)$ of a discrete group $G$ is defined by the cochain complex $C_b^k(G;\br)$ where
$$C_b^k(G;\br)=\{f\co G^k \to \br \, | \, sup_{G^k}|f(g_1,...,g_k)|<\infty \},$$
and the boundary $\delta\co C_b^k(G;A)\to C_b^{k+1}(G;A)$ is given by

$$\delta f(g_0,...,g_k)=f(g_1,...,g_k)+\sum_{i=1}^k (-1)^if(g_0,...,g_{i-1}g_i,...,g_k) + (-1)^{k+1}f(g_0,...,g_{k-1}).$$

See \cite{Gr82,Iva} as general references for bounded cohomology.

\begin{nota} Note that 1-cocycles are just group homomorphisms $f\co G \to \br$. In fact, $HOM(G)=H^1(G;\br)$. A \emph{quasicharacter} is an element $f\in C^1(G;\br)$ whose coboundary $\delta f$ lies in $C_b^2(G;\br)$. A \emph{pseudocharacter} is a quasicharacter such that $f(g^k)=kf(g)\; \forall \, k\in \bz$ and $g\in G$. 
\end{nota}

Let $\mathcal{V}(G)$ be the vector space of all quasi-homomorphisms $G\to \br$, $BDD(G)$ the subspace of all bounded functions.
Then, let $QH(G)=\mathcal{V}(G)/BDD(G)$.

There is an exact sequence
$$0\to H^1(G;\br)\to QH(G)\to H_b^2(G;\br)\to H^2(G;\br)$$

Using the sequence $f_1,f_2,...$ obtained in Proposition \ref{Prop2: BF}, Bestvina and Fujiwara prove that $[h_i]\in QH(G)$ is not a linear combination of $[h_1],...,[h_{i-1}]$, i.e., the sequence $[h_i]$ consists of linearly independent elements (see the proof of \cite[Theorem 1]{BF}). This implies that the dimenion of $H_b^2(G;\br)$ as a vector space over $\br$ is the cardinal of the continuum. See \cite[Corollary 1.3]{F1} and \cite[Theorem 1]{BF}.

Therefore, since the argument in Theorem \ref{Tma: bushy} works also for any pair $f_i,f_{i+1}$, the following is immediate.

\begin{cor} Let $G$ be a group acting on a (geodesic) Gromov hyperbolic graph $X$ satisfying the bottleneck property. If it is a Bestvina-Fujiwara action, then the dimension of the subspace generated by the bushy pseudocharacters as a vector space over $\br$ is the cardinal of the continuum. 
\end{cor}

\begin{cor} If a Cayley graph $X=\Gamma(G,S)$ satisfies the bottleneck property and the canonical action of the group is a Bestvina-Fujiwara action, then the dimension of the subspace generated by the bushy pseudocharacters as a vector space over $\br$ is the cardinal of the continuum. 
\end{cor}

\begin{cor} If Manning's action is a Bestvina-Fujiwara action, then the dimension of the subspace generated by the bushy pseudocharacters on $G$ as a vector space over $\br$ is the cardinal of the continuum. 
\end{cor}

In particular,

\begin{cor}\label{Cor: bushy} If there is a bushy pseudocharacter $h\co G \to \br$ then, the dimension of the subspace generated by the bushy pseudocharacters on $G$ as a vector space over $\br$ is the cardinal of the continuum. 
\end{cor}

This, together with Corolaries \ref{Cor: nonelementary0} and \ref{Cor: nonelementary} yields

\begin{cor} Consider a nonelementary action of a group $G$ on a quasi-tree $X$ and a nonelementary pseudocharacter $h\co G \to \br$. Then, if $h$ is bornologous on the action, the dimension of the subspace generated by the bushy pseudocharacters on $G$ (in particular, the dimension of $H_b^2(G;\br)$) as a vector space over $\br$ is the cardinal of the continuum. 
\end{cor}

\begin{cor} Consider a nonelementary action of a group $G$ on a quasi-tree $X$. If the action is metrically proper and there exist a nonelementary pseudocharacter then the dimension of the subspace generated by the bushy pseudocharacters on $G$ (in particular, the dimension of $H_b^2(G;\br)$) as a vector space over $\br$ is the cardinal of the continuum. 
\end{cor}


\begin{thebibliography}{99}

\bibitem{BF} M. Bestvina and K. Fujiwara \emph{Bounded cohomology of subgroups of mapping class groups}. Geometry and Topology \textbf{6} (2002) 69--89.


\bibitem{Bow} B. H. Bowditch. Notes on Gromov's hyperbolicity criterion. \emph{Group theory from a geometric viewpoint}. World Scientific, New Jersey, (1991) 64--167.


\bibitem{D1} M. Cencelj, J. Dydak, Z. Vavpeti\v{c} and \v{Z}. Virk. \emph{A combinatorial approach to Coarse Geometry}. arXiv:0906.1372v1. 


\bibitem{EF} D. B. A. Epstein and K. Fujiwara. \emph{The second bounded cohomology of word-hyperbolic groups}, Topology \textbf{36} (1997) 1275--1289.

\bibitem{F1} K. Fujiwara. \emph{The second bounded cohomology of a group acting on a Gromov-hyperbolic space}. Proc.  London Math. Soc. \textbf{3}, 76 (1998) 70--94.

\bibitem{F2} K. Fujiwara. \emph{The second bounded cohomology of an amalgamated free product of groups}. Trans. Amer. Math. Soc. \textbf{352}, 76 (2000) 1113--1129.

\bibitem{F-W} K. Fujiwara and K. Whyte. \emph{A note on spaces of asymptotic dimension one}. Algebraic and Geometric Topology \textbf{7} (2007) 1063-1070.

\bibitem{Gr82} M. Gromov. \emph{Volume and bounded cohomology}. Inst. Hautes Études Sci. Publ. Math. \textbf{56} (1982) 5--99.

\bibitem{Gr} M. Gromov. Hyperbolic groups, in \emph{Essays in group
theory.} Math. Sci. Res. Inst. Publ. 8, Springer-Verlag, New York,
(1987) 75--263.

\bibitem{Hei} J. Heinonen. \emph{Lectures on analysis on metric spaces}. Universitext. Springer-Verlag, New
York, 2001.

\bibitem{Iva} N. V. Ivanov. \emph{Foundations of the theory of bounded cohomology}. Zap. Nauchn. Sem. Leningrad. Otdel. Mat. Inst. Steklov. \textbf{143} (1985) 69--109. 

\bibitem{Man} J. F. Manning. 
\emph{Geometry of pseudocharacters.}
Geometry and Topology. \textbf{9}, (2005) 1147--1185.

\bibitem{Man2} J. F. Manning. 
\emph{Quasi-actions on trees and property (QFA)}. J. London Math. Soc. (2) \textbf{73}, (2006) 84--108.

\bibitem{M-P} A. Martínez-Pérez. \emph{Real valued functions and metric spaces quasi-isometric to trees.} Preprint.

\bibitem{MSW} L. Mosher, M. Sageev, K. Whyte. \emph{Quasi-actions on
trees I. Bounded valence.} Annals of Mathematics, \textbf{158},
(2003), pp 115--164.

\end{thebibliography}
\end{document}